\documentclass[letterpaper,oneside]{amsart}
\usepackage[english]{babel}
\usepackage{geometry,amssymb,mathrsfs,hyperref,xcolor}
\usepackage{mathtools}
\usepackage{amsmath}
\usepackage{amsthm}
\usepackage{amssymb}
\usepackage{thmtools}
\usepackage{hyperref}
\usepackage{tikz-cd}
\usepackage{comment}
\usepackage{enumitem}
\usepackage[backend=biber,url=false]{biblatex}
\usepackage{todonotes}
\usepackage[normalbf,normalem]{ulem}
\usepackage{calligra}
\usepackage{url}

\newcommand{\TeXmacs}{T\kern-.1667em\lower.5ex\hbox{E}\kern-.125emX\kern-.1em\lower.5ex\hbox{\textsc{m\kern-.05ema\kern-.125emc\kern-.05ems}}}
\newcommand{\assign}{:=}

\newcommand{\tmdummy}{$\mbox{}$}
\newcommand{\tmmathbf}[1]{\ensuremath{\boldsymbol{#1}}}

\newcommand{\tmop}[1]{\ensuremath{\operatorname{#1}}}

\theoremstyle{plain}
\newtheorem{theorem}{Theorem}[section]
\newtheorem{proposition}[theorem]{Proposition}
\newtheorem{lemma}[theorem]{Lemma}
\newtheorem{corollary}[theorem]{Corollary}
\theoremstyle{definition}
\newtheorem{definition}{Definition}[section]
\theoremstyle{remark}
\newtheorem{remark}[theorem]{Remark}
\newtheorem{example}[theorem]{Example}
\newtheorem{examples}[theorem]{Examples}

\numberwithin{equation}{section}

\newcommand{\Z}{\mathbb{Z}}

\newcommand{\Hom}{\mathbf{Hom}}
\newcommand{\catname}[1]{\textbf{#1}}

\newcommand{\Ob}{\mathbf{Ob}}
\newcommand{\cal}[1]{\ensuremath{\mathcal{#1}}}
\newcommand{\mc}[1]{\ensuremath{\mathcal{#1}}}
\newcommand{\s}[1]{\ensuremath{\mathscr{#1}}}
\newcommand{\site}[1]{\ensuremath{\tilde{\mathcal{#1}}}}
\newcommand{\cat}[1]{\ensuremath{\mathcal{#1}}}

\newcommand{\cov}{\ensuremath{\text{Cov}}}

\newcommand{\mcO}{\ensuremath{\cat{O}}}
\newcommand{\mcM}{\ensuremath{\cat{M}}}

\newcommand{\mcC}{\ensuremath{\cat{C}}}
\newcommand{\mcS}{\ensuremath{\cat{S}}}
\newcommand{\mcU}{\ensuremath{\cat{U}}}

\newcommand{\mcA}{\ensuremath{\cat{A}}}

\newcommand{\T}{\mathbb{T}}
\DeclareMathOperator{\sheafify}{\text{\calligra \footnotesize S\normalsize heaf}}
\DeclareMathOperator{\sheafhom}{\mathscr{H}\text{\kern -4pt {\calligra\large om}}\,}

\bibliography{/home/antonio-rieser/Bib/all}

\thanks{This material is based upon work supported by the US National Science
	Foundation under Grant No. DMS-1928930 while the author was in residence at the Simons Laufer Mathematical Sciences Institute (formerly MSRI) in Berkeley, California, as well as when the author participated in a Mathematical Sciences Research Institute program held in the summer of 2022 in partnership with the the Universidad
	Nacional Aut{\'o}noma de M{\'e}xico. This work was also supported by
	the SECIHTI Investigadoras y Investigadores por M{\'e}xico Project \#1076, the SECIHTI Ciencia de Fronteras grant CF-2019-217392, and
	by the grant N62909-19-1-2134 from the US Office of Naval Research Global and
	the Southern Office of Aerospace Research and Development of the US Air Force
	Office of Scientific Research.}

\title[Sheaf Theory for Data and Graphs: An Approach through \v Cech Closure Spaces]{Grothendieck Topologies and Sheaf Theory for Data and Graphs: An Approach through \v Cech Closure Spaces}

\author{Antonio Rieser}
\address{Área de Matemáticas Básicas, Centro de Investigación en Matemáticas, A.C., Calle Jalisco S/N, GTO, 36023, México}
\email{antonio.rieser@cimat.mx}

\begin{document}

\begin{abstract}
	We initiate the study of sheaves on \v Cech closure spaces, providing a new, unified approach to sheaf theory on many of the major classes of spaces of
	interest to applications: topological spaces, finite simplicial complexes (seen as $T_0$ topological spaces), graphs and digraphs (both seen as closure spaces), quivers (seen as a pair of closure spaces), and metric spaces decorated with
	a privileged scale, the latter of which are widely used in topological data analysis. Our construction proceeds by constructing a
	Grothendieck topology on the category $\cat{M}_{c_X}$ of finite intersections of subspaces of $(X,c_X)$ with non-empty $c_X$-interior, which is the natural generalization to closure spaces of the category $\mathcal{O}(X,\tau)$ of open sets in a topological space. 
	We continue by constructing the sheaf and \v Cech cohomologies on $\cat{M}_{c_X}$, and we then identify examples of non-topological closure spaces induced by graphs with non-trivial sheaf cohomology, in particular in dimension two.
\end{abstract}



\maketitle

\section{Introduction}

Starting from its origins in algebraic topology, sheaf theory has become an
indispensable part of homological algebra, and it has many important
applications in fields as diverse as algebraic geometry and partial
differential equations. In the past decade, particularly with the rise of
topological data analysis, there has been increased interest in extending the
reach of sheaf theory to scientific and engineering applications,
and a number of intriguing efforts have been made in that direction
{\cite{Baryshnikov_Ghrist_2009,Curry_Ghrist_Robinson_2012,Curry_2015,Ghrist_Krishnan_2013, Hylton_etal_2020, Hylton_etal_2022,Robinson_2012,Robinson_2017,Robinson_2019, Short_etal_2022, Short_etal_2021}}.
The most common construction of this kind in the literature involves
studying sheaves on a combinatorially defined space which may be given a topology, such as a simplicial
complex whose set of simplices is endowed with a $T_0$ topology, either in the version on cell complexes first
developed in {\cite{Shepard_1985}} and revived in {\cite{Curry_2014}}, where
the $T_0$ topology need not be used explicitly to develop much of the theory, or else on spaces of posets as in {\cite{Robinson_2017}}, where the $T_0$ topology takes on a more central role. Related constructions of sheaf theory on simplicial complexes can also be found in Section 8.1 of
{\cite{Kashiwara_Schapira_1994}}.

In this article, we introduce a new approach to the construction of sheaf
theory on discrete and combinatorial spaces by constructing sheaves on {\v
C}ech closure spaces, a category which contains the major classes of spaces of interest to applications: topological spaces, including finite simplicial complexes with a $T_0$ topology on the collection of simplices, graphs and digraphs (where the vertices are seen as a closure space and the closure structure is induced by the edges), quivers (where the vertices and edges are given compatible closure structures), and metric spaces
decorated with a privileged, non-zero scale, which induces a special class of semi-pseudometric spaces of interest to topological data analysis. The construction given here is a generalization to closure spaces of sheaf theory on the standard Grothendieck topology on topological spaces, and, as such, also includes cellular sheaves \cite{Curry_2014, Curry_2015} as a special case. Unlike in the applications of cellular sheaves to
point clouds, however, our construction allows one to define sheaves
directly on the set of points, thus eliminating the need to pass to an ancillary
simplicial complex. For quivers, Grothendieck topologies have been considered on the path category of a quiver in \cite{Murfet_2005}, and a Grothendieck category was created from colored quivers in 
\cite{Kanda_2020} in order to answer an algebraic question about categories with enough compressible objects. Both constructions are quite different from the one pursued here, however, as the Grothendieck topologies in these papers are constructed on auxiliary objects built using the quivers, and not on the quivers themselves. Our construction also provides, to the best of our knowledge, the first construction of sheaf theory for directed graphs, since cellular sheaf theory does not apply to this case.

The algebraic topology of {\v C}ech closure
spaces appears to have been first studied in {\cite{Demaria_1987,Demaria_Garbaccio_1984_2}}, where it was used to construct new invariants of graphs and digraphs. After a long period of inactivity, this idea was taken up again
in {\cites{Rieser_2021,Rieser_arXiv_2022}}, where the
development of homotopy theory on {\v C}ech closure spaces was continued, and where it was shown that
this framework also covers the case of metric spaces which are endowed with a
preferred scale, the scale indicating the minimal radius of a metric ball around a point which must be
contained in any neighborhood of that point. Further work on the algebraic topology of closure spaces led to a unified view of a number of discrete homotopy theories in
{\cite{Bubenik_Milicevic_2024}}, where it was shown that many of the
discrete homotopy theories on studied on graphs
{\cite{Barcelo_Capraro_White_2014,Babson_etal_2006,Grigoryan_etal_2014}}
can be expressed as a homotopy theory on {\v C}ech closure spaces
simply by changing the cylinder functor and the product used to define
homotopy. Although not mentioned in {\cite{Bubenik_Milicevic_2024}},
both digital homotopy {\cite{Ayala_etal_2003,Boxer_2005}} and the discrete
homotopies studied in {\cite{Plaut_Wilkins_2013}} may also be expressed in
terms of homotopies on appropriate closure spaces as well. It was additionally shown in \cite{Bubenik_Milicevic_2024} that 
stability theorems for persistent versions of a functor (persistent homology, persistent cohomology, persistent homotopy, etc) follow directly from 
the closure space homotopy invariance of that functor, demonstrating the importance of the algebraic topology of closure spaces to more traditional
approaches to applied
topology. Finally, the category of pseudotopological spaces, the Cartesian closed hull of both the category of topological spaces and \v{C}ech closure spaces, was shown to admit a Quillen-type model structure in \cite{Rieser_arXiv_2022}, and this model structure was further shown in \cite{Ebel_Kapulkin_arXiv_2023} to be Quillen equivalent to the Quillen model structure on topological spaces.

There are a number of obstacles which must be overcome in order to develop sheaf theory on general {\v C}ech closure spaces. First, in most
interesting closure spaces, there are simply not enough open sets for the
classical, topological sheaf cohomology to provide much information. We see
this already with the closure space $(V, c_G)$ induced by a connected graph $G
= (V, E)$. In this case, the open sets of $(V,c_G)$ are simply the indiscrete
topology on the vertex set, i.e. the topology consisting only of $\{\varnothing,
V\}$, and, like the topology, the resulting sheaf theory is trivial. In closure spaces, however, the interior of a set
may be non-open, and there are many more non-open sets with non-trivial (and non-open) interiors than there are open sets. Using these, we may form so-called
interior covers, which provides an apparent solution to the problem of the lack of open sets.
However, this leads to a second issue, which is that the
category $\cat{N}_{c_X}$ of subsets of $(X,
c_X)$ with non-empty interior is not closed under finite intersections. Naively, one may try to simply add the missing sets to $\cat{N}_{c_X}$, and so we might posit that the category in which
we should work is not $\cat{N}_{c_X}$, but rather the category which
contains $\cat{N}_{c_X}$ together with all of the finite intersections of its elements, which we denote by $\cat{M}_{c_X}$. However, while the sets in
$\cat{M}_{c_X}$ now form a topology on $X$, this topology is often too
fine to be interesting. That is, if we build sheaves using open covers of $X$ in the topology given by $\cat{M}_{c_X}$, then much of the structure we wish to capture may be lost. We instead consider a kind of hybrid of these two options. We begin with presheaves defined on the category  whose objects are the subspaces of $(X,c_X)$ induced by inclusions of the sets $\cat{M}_{c_X}$ into $(X,c_X)$, and whose morphisms are inclusions of subspaces. (We also call this category $\cat{M}_{c_X}$ in a slight abuse of notation.) We then use interior covers of the subspaces in $\cat{M}_{c_X}$ to form a site, which, in turn, generates our Grothendieck topology. 

After constructing the Grothendieck
topology on $\cat{M}_{c_X}$, the construction of sheaf and {\v C}ech cohomology on closure spaces follows from general considerations, with the \v Cech cohomology being a generalization of  the construction in \cite{Palacios_Lic_2019} for constant coefficients. We appeal to several results which follow from the convergence of the spectral sequence from {\v
C}ech cohomology to sheaf cohomology on a Grothendieck topology in order to produce examples of non-topological closure
spaces with non-trivial sheaf cohomology. In particular, we produce a collection of closure spaces induced by graphs which we show to have non-trivial sheaf cohomology in dimension two, demonstrating the non-topological nature of the cohomology.

\section{{\v C}ech Closure Spaces}\label{sec:Closure spaces}

In this section, we introduce \v Cech closure spaces, show how topological spaces, graphs, and metric spaces with a privileged scale all form examples of closure spaces, how quivers induce induce pairs of topological spaces, and we collect several basic results about closure spaces which we will use later.

\begin{definition}
  Let $X$ be a set, and let $c : \mathcal{P} (X) \to \mathcal{P} (X)$ be a map
  on the power set of $X$ which satisfies
  \begin{enumerate}
    \item $c (\varnothing) = \varnothing$
    
    \item $A \subset c (A)$ for all $A \subset X$
    
    \item $c (A \cup B) = c (A) \cup c (B)$ for all $A, B \subset X$
  \end{enumerate}
  The map $c$ is called a \emph{{\v C}ech closure operator} (or \emph{closure operator}) on $X$, and the pair $(X, c_X)$ is called a
  \emph{{\v C}ech closure space} (or {\emph{closure space}}).
  
  A function $f:c_X \to c_Y$ is said to be \emph{continuous} iff $f(c_X(A)) \subset c_Y(f(A))$ for every $A\subset X$. 
\end{definition}

\begin{examples}
  \label{ex:Closure spaces}{\tmdummy}
  
  \begin{enumerate}[label=(\alph*),wide,labelwidth=!]
    \item Let $X = \varnothing$, and let $c_\varnothing$ be defined by $c_\varnothing(\varnothing)= \varnothing$. Then
    $(\varnothing,c_\varnothing)$ is a closure space, which we call the \emph{empty closure space}.
    \item If $X$ is a set and $c(A) = A$ for all $A\subset X$, then $c$ is the \emph{discrete closure structure} on $X$. Conversely, if $c(A) = X$ for all $A\subset X$, then $X$ is the \emph{indiscrete closure structure} on $X$.
    \item \label{item:Example topological closure}Let $(X, \tau)$ be a
    topological space with topology $\tau$. For any $A \subset X$, denote by
    $\bar{A}$ the topological closure of $A$. Then $c_{\tau} (A) = \bar{A}$
    is a {\v C}ech closure operator. Note that, in this case, $c_{\tau}^2 (A)
    = c_{\tau} (A)$. Closure operators $c$ with the property that $c^2 = c$
    are called {\emph{Kuratowski}} or {\emph{topological closure operators}},
    and it can be shown that, for Kuratowski closure operators, the collection
    \[ \mathcal{O}(X) \assign \{X \setminus c (A) \mid A \subset X\} \]
    forms the open sets of a topology on $X$. (See {\cite{Cech_1966}}, Theorem
    15.A.2(a) for a proof.) Furthermore, a map $f:(X,\tau) \to (Y,\tau')$ between topological spaces
    is topologically continuous iff it is continuous as a map $f:(X,c_\tau) \to (Y,c_{\tau'})$  
    between the induced closure spaces. (\cite{Cech_1966}, Theorem 16.A.10)
    
    \item \label{item:Mesoscopic space} Let $(X, d)$ be a metric space, and $r \geq 0$ a non-negative real
    number. For any $A \subset X$, define
    \begin{equation}\label{eq:Closure with scale} c_r (A) \assign \{x \in X \mid d (x, A) \leq r\} . \end{equation}
    Then $c_r$ is a closure operator on $X$. For $r = 0$, $(X, c_0)$ is
    toplogical by the discussion in Example \ref{item:Example topological closure} above, and if $r > 0$, we call
    $(X, c_r)$ a {\emph{mesoscopic space}}. Functions between closure spaces of the form $(X,c_p)$ and $(Y,c_q)$ admit the following convenient characterization of continuity, as shown in \cite{Rieser_2021}.
    
    \begin{proposition}[\cite{Rieser_2021}, Proposition 3.5]
    	\label{prop:p-q cont}
    	Let $p,q\geq 0$ be non-negative real numbers, and let $(X,d_X)$ and $(Y,d_Y)$ be metric spaces with closure structures $c_p$ and $c_q$, respectively, where the $c_p$ and $c_q$ are as in Equation (\ref{eq:Closure with scale}). Then a function $f:(X,c_p)\to
    	(Y,c_q)$ is continuous iff for every $\epsilon > 0$ and every $x \in X$, there exists a $\delta_{x,\epsilon} > 0$ such that
    	\begin{equation} \label{eq:p-q cont}
    	d_X(x,x') < p + \delta_{x,\epsilon} \implies d_Y(f(x),f(x')) < q + \epsilon.
    	\end{equation}
    \end{proposition}

\begin{definition}
	When a function $f:(X,d_X) \to (Y,d_Y)$ between metric spaces satisfies Equation (\ref{eq:p-q cont}) in \autoref{prop:p-q cont} above, we say that $f$ is \emph{$(p,q)$-continuous}.
\end{definition}

    \item Let $G = (V, E)$ be a reflextive graph with vertices $V$ and edges $E$. (Recall that a graph $G= (V,E)$ is reflexive iff, for each vertex $v$ 
    there is an edge $(v,v) \in E$.) We
    define a closure operator $c_G : \mathcal{P} (V) \to \mathcal{P} (V)$ in
    the following way. First, let $s : V \to \mathcal{P} (V)$ be the \emph{star} of  a vertex $v$, i.e. the map
    \[ s (v) \assign \{v' \in V \mid (v, v') \in E\} . \]
    For an arbitrary $A \subset V$, we now define the operator $c_G$ by
    \[ c_G (A) = \bigcup_{v \in A} s(v) . \]
    Then $c_G$ is a closure operator, and $(V, c_G)$ is a closure space, which
    we call the {\emph{closure space induced by the graph $G$}}. Addtionally, a map $f:(V,c_G) \to (V',c_{G'})$ is continuous
    iff $f$ is a graph homomorphism. (This follows directly from \autoref{prop:p-q cont} by taking $p=q=1$ and viewing a graph as a metric space with the shortest path metric, noting that $c_1 = c_G$.) 
    
    Given a non-reflexive graph $G = (V,E)$, we construct the closure space $(V,c_G)$ by first adding the diagonal elements $(v,v) \subset V \times V$ to $E$, and then defining $c_G$ with respect to the graph with all edges $(v,v)$ added.  
      
    \item\label{ex:Quiver 1} Recall that a \emph{quiver} is a directed graph which possibly contains multiple edges between any two vertices. We formalize this defining a quiver as a quadruple $Q = (V,E,s,t)$ where $V$ is the set of vertices, $E$ is the set of edges (but E is not necessarily a subset of $V \times V$, unlike in a graph), and $s,t:E \to V$ are the \emph{source} and \emph{target} maps, respectively. A quiver induces a directed graph $G = (V,E_Q)$ on the vertices $V$ by $(v,v') \in E_Q$ iff $\exists e \in E$ with $s(e)= v$ and $t(e) = v'$. Let $(V,c_G)$ denote the induced closure space. We then define closure structures $c_s,c_t: \cat{P}(E) \to \cat{P}(E)$ on $E$ by
    \begin{align*}
    	c_s(A) & = s^{-1}(c_G(s(A))),\\
    	 c_t(A) & = t^{-1}(c_G(t(A))),
    \end{align*}
    for any $A \subset E$. We now define $c_Q(A) \coloneqq c_s(A) \cap c_t(A)$. Then $c_Q$ is the closure structure projectively generated by the mappings $s,t$ (\cite{Cech_1966}, Theorem 32 A.4), and $s,t:(E,c_Q)\to (V,c_G)$ are continuous as maps of closure spaces. 
    \item\label{ex:Quiver 2} A second closure structure on the edges $E$ of a quiver $Q = (V,E,s,t)$ may be constructed as follows. As before, we generate the directed graph $G = (V,E_Q)$ and the induced closure space $(V,c_G)$ from the quiver. We then define $c_{Q,-}:\cat{P}(E) \to \cat{P}(E)$ as follows. First, let 
    \begin{align*}
    	c_{s,-}(A) &= A \cup s^{-1}(c_G(s(A)) - s(A))\\
    	c_{t,-}(A) &= A \cup t^{-1}(c_G(t(A)) - t(A)),
    \end{align*}
    and define $c_{Q,-}(A) \coloneqq c_{s,-}(A) \cap c_{t,-}(A)$.
    One may then confirm that $s,t:(E,c_{Q,-}) \to (V,c_G)$ is continuous, as $c_{Q,-}$ is finer than the structure $c_Q$ in the example above. Unlike $c_Q$, in the structure $c_{Q,-}$, the closure of an edge does not include the other edges with the same source and target, which may be of interest when using these spaces to investigate different paths through the quiver.
    \end{enumerate}
\end{examples}

\begin{definition}
  Let $(X, c_X)$ be a closure space. For any $A \subset X$, we define the
  {\emph{interior of A}} by
  \[ i_X (A) \assign X - c_X (X - A), \]
  and we say that $U \subset X$ is a {\emph{neighborhood of $A$}} iff $A
  \subset i_X (U)$. We will call $i_X : \mathcal{P} (X) \to \mathcal{P} (X)$ the
  {\emph{interior operator of $(X, c_X)$.}}
\end{definition}

The next three propositions enumerate the essential properties of interior operators.

\begin{proposition}
  [{\cite{Cech_1966}}, 14 A.11]\label{prop:Interior properties} Let $(X, c_X)$
  be a closure space. The interior operator $i_X : \mathcal{P} (X) \to
  \mathcal{P} (X)$ sat{\emph{}}isfies the following
  \begin{enumerate}
    \item $i_X (X) = X$
    
    \item $i_X (A) \subset A \text{ for all } A \subset X$
    
    \item \label{item:Interior 3}$i_X (A \cap B) = i_X (A) \cap i_X (B)  \text{ for
    all } A, B \subset X$
  \end{enumerate}
\end{proposition}

The following two corollaries follow directly from Item (\ref{item:Interior 3})

\begin{corollary}
	\label{prop:Interior of subsets}Let $(X, c_X)$ be a {\v C}ech closure space,
	and suppose that $B \subset A \subset X$. Then $i_X (B) \subset i_X (A)$.
\end{corollary}

\begin{corollary}
	\label{cor:Intersection of neighborhoods}Let $(X, c_X)$ be a {\v C}ech
	closure space. For any two neighborhoods $U$ and $V$ of a point $x \in X$,
	the intersection $U \cap V$ is also a neighborhood of $x$.
\end{corollary}

\begin{proof}
	By hypothesis, $x \in i_X (U) \cap i_X (V)$, and by Item (3) of Proposition
	\ref{prop:Interior properties}, $i_X (U) \cap i_X (V) = i_X (U \cap V)$.
\end{proof}

The following definition describes the subsapce closure structure on a subset of a closure space.

\begin{definition}
	Given a closure space $(X,c_X)$ and a subset $U \subset X$, we define
	the \emph{subspace closure operator} by $c_U(A) \coloneqq c_X(A) \cap U$ for all $A \subset U$. The closure space $(U,c_U)$ is a \emph{subspace of $(X,c_X)$}, which we denote $(U,c_U) \subset (X,c_X)$.
\end{definition}

We will also need the following result on neighborhoods in subspaces, which we quote from \cite{Cech_1966}.

\begin{proposition}[\cite{Cech_1966}, Theorem 17.A.9(a,b)]
	\label{prop:Interior of subspaces}
	Let $(X,c_X)$ be a \v Cech closure space~, and let $(U,c_U) \subset (X,c_X)$ be a subspace of $(X,c_X)$.
	\begin{enumerate}[wide,labelindent=0pt,leftmargin=*]
	\item For any $V \subset U$, $i_{U}(V) = i_X(V \cup (X-U)) \cap U$, where $i_{U}$ and $i_X$ are the interior operators for $(U,c_U)$ and $(X,c_X)$, respectively.
	\item \label{item:Subspace neighborhoods} For any $x \in U$, a set $V \subset U$ is a neighborhood of $x$ in $(U,c_U)$ iff there exists a neighborhood $W$ of $x$ in $(X,c_X)$ such that $V = W\cap U$.
	\end{enumerate}
\end{proposition}

We now discuss the different kinds of covers on closure spaces which we will use to construct our Grothendieck topology.

\begin{definition}
	\label{def:Covers} 
	Let $(X, c_X)$ be a closure space.
  \begin{enumerate}[wide,labelwidth=!,labelindent=0pt]
    \item We say that a collection $\mathcal{U}$ of subsets of $X$ is a
    \emph{cover} of $(X, c_X)$ iff
    \[ X = \bigcup_{U_{\alpha} \in \mathcal{U}} U_{\alpha}, \]
    \item \label{item:Interior cover} We say that a collection $\mathcal{U}$ of subsets of $X$ is an
    \emph{interior cover} of $(X, c_X)$ iff
    \begin{equation*}
      X = \bigcup_{U_{\alpha} \in \mathcal{U}} i_X (U_{\alpha}) .
    \end{equation*}
  \end{enumerate}
\end{definition}

An illustrative example of an interior cover on a non-topological closure space is given by the following cover of a closure space induced by a graph.

\begin{example}
	Let $G=(V,E)$ be a (possibly directed) reflexive graph, and let $(V,c_G)$ be the closure space induced by the graph. Let \begin{equation*}
		S(v)\assign \{v'\in V\mid (v',v)\in E\}.
	\end{equation*} 
	Then $\mathcal{S}_G\assign\{S(v)\}_{v \in V}$ is an interior cover on $(V,c_G)$. Note that, when $G$ is undirected, $S(v)$ is the star of the vertex $v$, and when $G$ is directed, $S(v)$ is the star of the vertex $v$ in the graph $G'=(V,E^{-1})$, where $E^{-1} = \{(v',v) \mid (v,v') \in E \}$. Also, this interior cover is maximal among interior covers of $(V,c_G)$ with respect to the preorder given by refinement, i.e. $\mathcal{U} \prec \mathcal{V}$ iff $\mathcal{V}$ refines $\mathcal{U}$. That is, for any interior cover $\mathcal{U}$ of $(V,c_G)$, $\mathcal{S}_G$ refines $\mathcal{U}$.
\end{example}

In \autoref{prop:Interior cover of interior covers} below, we show that 
the union of the interior covers of a collection of subspaces $(U_\alpha,c_{U_\alpha})$, $\alpha \in I_\mathcal{U}$, is an interior cover of $(X,c_X)$ when the set $\mathcal{U} \assign \{U_\alpha \mid \alpha \in I_\mathcal{U}\}$ is an interior cover of $(X,c_X)$. This will be necessary for the construction of our Grothendieck topology. We begin with the following lemma.

\begin{lemma}
  \label{lem:Nesting}Let $(X, c_X)$ be a Cech closure space, and let $A
  \subset X$. Suppose that a collection $\mathcal{U}$ of subsets of $A$ is an
  interior cover of $(A, c_A)$, where $c_A$ is the subspace closure structure
  on $A \subset (X, c_X)$, i.e. $c_A(B) = A \cap c_X(B)$ for any $B \subset A$. Then
  \[ i_X (A) = \bigcup_{U_{\alpha} \in \mathcal{A}} i_X (U_{\alpha}), \]
  where $i_X$ is the interior operator on $(X, c_X)$.
\end{lemma}
\begin{proof}
	Let $U \subset A$, then by \cite{Cech_1966}, Theorem 17.A.9(a), $i_A(U) = A \cap i_X(U \cup (X-A))$. Writing $W \coloneqq X-A$,
	it follows that 
	\begin{equation}
		\label{eq:Interiors}
	\begin{aligned}
		i_X(A) \cap i_A(U)& =  i_X(A) \cap i_X(U \cup W) \cap A
		= i_X(A \cap (U \cup W)) \\
		&= i_X(A \cap U) 
		= i_X(U).
	\end{aligned}
	\end{equation}
 Since $\mathcal{U}$ is an interior cover of $A$, however, we have that
  \begin{align*}
    i_X (A) & =  i_X (A) \cap A
    =  i_X (A) \cap \left( \bigcup_{U \in \mathcal{U}} i_A
    (U) \right)\\
    & =  \bigcup_{U \in \mathcal{U}} i_X (A) \cap i_A
    (U)
     = \bigcup_{U \in \mathcal{U}} i_X (U),
  \end{align*}
  as desired, where the final equality follows from Equation \ref{eq:Interiors}.
\end{proof}

\begin{proposition}
	\label{prop:Interior cover of interior covers}
	Let $(X,c_X)$ be a closure space, and let $\mathcal{U} \assign \{U_{\alpha} \mid \alpha \in I_{\mathcal{U}}\}$ be an interior cover of $(X,c_X)$, where $I_{\mathcal{U}}$ is an index set. For each set $U_{\alpha} \in \mathcal{U}$, let $\mathcal{U}_\alpha \assign \{U_{\alpha\beta} \mid \beta \in I_{\alpha}\}$, be an interior cover of the subspace $(U_{\alpha},c_{U_{\alpha}}) \subset (X,c_X)$, where for each $\alpha \in I_{\mathcal{U}}$, $I_{\alpha}$ is the index set for the cover $\mathcal{U}_\alpha$. Then the collection of sets 
	\[\mathcal{V} \assign \bigcup_{\alpha \in I_{\mathcal{U}}} \mathcal{U}_{\alpha}\]
	is an interior cover of $(X,c_X)$. 
\end{proposition}

\begin{proof}
	We compute
	\begin{align*}
		\bigcup_{V \in \mathcal{V}} i_{X}(V) = \bigcup_{\alpha \in 
				I_{\mathcal{U}}} \bigcup_{\beta \in I_{\alpha}} 
				i_X(U_{\alpha\beta}) 
		= \bigcup_{\alpha \in I_{\mathcal{U}}} i_X(U_{\alpha}) 
		= X,
	\end{align*}
	where the second equality follows from \autoref{lem:Nesting}. Therefore $\mathcal{V}$ is an interior cover of $(X,c_X)$.
\end{proof}

\section{Sheaves on \v Cech Closure Spaces}

In this section, we review the definitions of sites and Grothendieck topologies on a category, following the presentations in \cite[\href{https://stacks.math.columbia.edu/tag/00VG}{Section 00VG},\href{https://stacks.math.columbia.edu/tag/00YW}{Section 00YW},\href{https://stacks.math.columbia.edu/tag/00ZB}{Section 00ZB} ]{stacks-project}, {\cite{Kashiwara_Schapira_2006}}, and \cite{MacLane_Moerdijk_1994}, and we construct a natural site on the category $\cat{M}_{c_X}$, defined below, given a closure space $(X,c_X)$. We then recall that how a site generates a Grothendieck topology, and we define the canonical Grothendieck topology on a closure space. Note that there is some discrepancy in the terminology used in the literature. What is called a \emph{site} in \cite{stacks-project} is called a \emph{category endowed with a pretopology} in \cite[Exposé II, Definition 1.3]{Artin_etal_1972}, a \emph{family of coverings} in \cite{Kashiwara_Schapira_2006},  \emph{the family of coverings of a topology} in \cite{Artin_1962}, a \emph{basis for a Grothendieck topology} in \cite{MacLane_Moerdijk_1994}, and a \emph{topology} in \cite{Tamme_1994}. We follow the terminological conventions of \cite{stacks-project}.

\subsection{Sites on the category $\cat{M}_{c_X}$}

	 We begin by recalling the definition of a site on a category $\mcC$ with pullbacks, and for any object $U \in \mcC$, denote by $\mcC_U$ the slice category of $\mcC$ over $U$.

\begin{definition}
	\label{def:Site}
	A {\emph{site}} $\site{C} \coloneqq (\cat{C},\cov{C})$ is given by a category $\mcC$ and, for every $U \in \mcC$, a
	collection $\site{C}(U)$ of families of objects in $\mcC_U$ which satisfies
	\begin{enumerate}
		\item \label{item:Site 1}For any isomorphism $U' \to U$, the family
		consisting of the single morphism $\{U' \to U\}$ is a member of $\site{C}(U)$.
		
		\item \label{item:Site 2}If $\{f_{\alpha} : U_{\alpha} \to U \mid \alpha
		\in A\} \in \site{C}(U)$, then for any morphism $g : V \to U$, the family
		$\{U_{\alpha} \times_U V \to V\mid \alpha \in A\} \in \site{C} (V)$.
		
		\item \label{item:Site 3}If $\{f_{\alpha} : U_{\alpha} \to U \mid \alpha
		\in A\} \in \site{C}(U)$ and, for every $\alpha \in A$, there is a family
		$\{g_{\alpha \beta} : V_{\alpha \beta} \to U_{\alpha} \mid \beta \in
		B_{\alpha} \} \in \site{C}(U_{\alpha})$, then the family $\{f_{\alpha} \circ
		g_{\alpha \beta} : V_{\alpha \beta} \to U \mid \alpha \in A, \beta \in
		B_{\alpha} \} \in \site{C}(U)$.
	\end{enumerate}
	The elements of each $\site{C}(U)$ are called \emph{coverings of $U$ in the site $\site{C}$}, or simply \emph{coverings of $\site{C}$}. We refer to the total collection of coverings of $\site{C}$ by $\text{Cov}(\site{C}) \coloneqq \{\site{C}(U)\mid U \in \cat{C}\}$. We will say \emph{\site{C} is a site on \cat{C}} to indicate that $\site{C}$ is the site $\site{C} = (\cat{C},\cov{C})$. When it is clear from context, we will sometimes refer to both the site and the category by the same symbol.
\end{definition}

\begin{definition}\label{def:Morphism of sites}
	Let $\site{C}$ and $\site{D}$ be sites on $\cat{C}$ and $\cat{D}$, respectively. A \emph{morphism of sites} $f:\site{C} \to \site{D}$ is a functor $f:\cat{C} \to \cat{D}$ such that
	\begin{enumerate}
		\item $\{U_i \xrightarrow{\phi_i} U\} \in \cov(\cat{C})$ implies $\{f(U_i) \xrightarrow{\phi_i} f(U)\} \in \cov(\cat{D})  $,
		\item For any $\{U_i \to U\}_{i \in I} \in \cov(\cat{C})$ and $V\to U$ a morphism in $\cat{C}$, the canonical morphism (from the pullback diagram)
		\[ f(U_i \times_U V) \to f(U_i) \times_{f(U)} f(V)
		\]
		is an isomorphism for every $i \in I$.
	\end{enumerate}
\end{definition}

Given a closure space $(X,c_X)$, we now define the category $\cat{M}_{c_X}$ over which we will build our site.

\begin{definition} Let $(X, c_X)$ be a \v Cech closure space. 
		 Define $\cat{M}_{c_X}$ to be the category whose objects are all subspaces $(V,c_V) \subset (X,c_X)$ such that $V$ is the intersection of a finite number of subsets of $(X,c_X)$ with
		 non-empty $c_X$-interior, and whose morphisms are the inclusion maps between subspaces, i.e. 
		 \begin{align*}
		 	\Ob(\cat{M}_{c_X}) \coloneqq & \left\{ (U,c_U) \subset (X,c_X) \,\middle| \,
		 	\exists \{V_k \subset X \mid i_X(V_k) \neq \varnothing\}_{k=1}^n,
		 	U = \cap_{k=1}^n V_k \right\} \\	
		 	\Hom_{\cat{M}_{c_X}}(V,U)& \coloneqq \begin{cases} \left\{ \iota^V_U:(V,c_V) \hookrightarrow (U,c_U) \right\}, &V \subset U,\\
		 		\varnothing, & V\not \subset U.
		 	\end{cases}
		 \end{align*}
		  Note that $\cat{M}_{c_X}$ includes the empty closure space $(\varnothing,c_\varnothing)$ iff there are two sets with non-empty $c_X$-interior whose intersection is empty.
\end{definition}

We now proceed to constructing the site on $\cat{M}_{c_X}$. We first define interior covers of elements of $\cat{M}_{c_X}$ in terms of morphisms in $\cat{M}_{c_X}$, in order to align our terminology with that of a site.

\begin{definition}
	Let $(X,c_X)$ be a \v Cech closure space, and let $\cat{M}_{c_X}/c_U$ denote the slice category of $\cat{M}_{c_X}$ over $(U,c_U)$. Suppose that $\mathcal{U}$ is a collection of morphisms 
	\[ \left\{\phi^{U_\alpha}_{U}:(U_{\alpha},c_{\alpha})\rightarrow (U,c_U) \right\}_{ \alpha \in A} \subset \Ob(\cat{M}_{c_X}/c_U)\]
	 in $\cat{M}_{c_X}/c_U$, where $A$ is an index set. We say that $\mathcal{U}$ is an \emph{interior cover of $(U,c_U) \in \cat{M}_{c_X}$} iff 
	\begin{equation*}
		U = \bigcup_{\phi^{U_\alpha}_{U} \in \mathcal{U} } i_U\left(\phi^{U_\alpha}_{U} (U_\alpha)\right),
	\end{equation*}
	where $i_U$ is the interior operator of the subspace $(U,c_U) \subset (X,c_X)$.
\end{definition}
The following lemma shows that the collection of interior covers of subspaces in $\cat{M}_{c_X}$ are closed under certain operations.

\begin{lemma}
	\label{prop:Conditions on covers}Let $(X, c_X)$ be a {\v C}ech closure space
	with interior operator $i_X$, and let $(U,c_U) \in \cat{M}_{c_X}$.
	\begin{enumerate}
		\item \label{item:Conditions on covers 1}Let  $\{\phi^{U_\alpha}_{U}:U_{\alpha} \to U\}_{\alpha \in
			A} \subset \cat{M}_{c_X}/c_U$ be an interior cover of $(U, c_U)$, and let $(V,c_V) \in \cat{M}_{c_X}$, $V
		\subset U$. Then $\{\phi^{V \cap U_{\alpha}}_{V}:V \cap U_{\alpha} \to V \}_{\alpha \in A}$ is an interior
		cover of $(V, c_V)$.
		
		\item \label{item:Conditions on covers 2}Let $\{\phi^{U_\alpha}_{U}:U_{\alpha} \to U\}_{\alpha \in
			A} \subset \cat{M}_{c_X}/c_U$ be an interior cover of $(U, c_U)$, and, for each $\alpha \in A$, let
		$\{\phi^{V_{\beta \alpha}}_{U_\alpha} : V_{\beta \alpha} \to U_\alpha \}_{\beta \in B_{\alpha}}\subset \cat{M}_{c_X}/c_{U_\alpha}$ be an interior cover of the subspace
		$(U_{\alpha},c_\alpha)$. Then $\{V_{\beta \alpha} \}_{\alpha \in A, \beta \in B_\alpha}$ is
		an interior cover of $(U,c_U)$.
	\end{enumerate}
\end{lemma}

\begin{proof}
	\begin{enumerate}[wide,labelwidth=!,labelindent=0pt]
		\item Follows from the definition of interior covers, Definition \ref{def:Covers}(\ref{item:Interior cover}), and \autoref{prop:Interior of subspaces}(\ref{item:Subspace neighborhoods}).
		\item From \autoref{lem:Nesting}, for each $\alpha \in A$ we have 
		\begin{equation*}
			i_U(U_\alpha) = \bigcup_{\beta \in B_\alpha} i_U(V_{\beta\alpha}),
		\end{equation*}
		and therefore 
		\begin{equation*}
			U = \bigcup_{\alpha \in A} i_U(U_\alpha) = \bigcup_{\alpha \in A} \bigcup_{\beta \in B_\alpha} i_U(V_{\beta\alpha}),
		\end{equation*}
		as desired.\qedhere
	\end{enumerate} 
\end{proof}

We now observe that the category $\cat{M}_{c_X}$ admits pullbacks.

\begin{lemma}
	\label{lem:Pullbacks}
	Let $(X,c_X)$ be a \v Cech closure space. 
	The category $\cat{M}_{c_X}$ admits pullbacks.  Furthermore, the pullback $(V \times_U W,c_{V\times_U W})$ in the diagram
	\begin{equation*}
		\begin{tikzcd}
			(V \times_U W,c_{V\times_U W}) \ar[d] \ar[r] 
			\arrow[dr, phantom, "\scalebox{1.5}{$\lrcorner$}" , very near start]& (V,c_V) \ar[d] \\
			(W,c_W) \ar [r]& (U,c_U)
		\end{tikzcd}
	\end{equation*}
	in $\cat{M}_{c_X}$ is given by $(V \times_U W,c_{V\times_U W}) = (V \cap W,c_{V \cap W})$.
\end{lemma}

\begin{proof}
	We note that the forgetful functor $F:\cat{M}_{c_X}\to \catname{Set}$ is fully faithful, where $F$ sends each subspace $(U,c_U) \in \cat{M}_{c_X}$ to its underlying set $U$ and is the identity on morphisms. Furthermore, in $\catname{Set}$, the pullback $V \times_U W = V \cap W$. Since $V,W \in \cat{M}_{c_X}$ implies that 
	$V \cap W \in \cat{M}_{c_X}$, the result follows, using the subspace structure $c_{V\cap W}$ on $V \cap W$.
\end{proof}

Combining the above lemmas, we are now ready to construct our site on $\cat{M}_{c_X}$.

\begin{theorem}
	\label{thm:Grothendieck site full}
	Let $(X,c_X)$ be a closure space, and
 	For each $(U,c_U) \in \cat{M}_{c_X}$, define 
	\begin{align*}
		\cov(U,c_U) & \coloneqq \{\mathcal{U} \subset \cat{M}_{c_X}/c_U \mid \mathcal{U} \text{ is an
			interior cover of } (U, c_U) \}, \text{and} \\
			\cov(\cat{M}_{c_X}) & \coloneqq \{ \cov(U,c_U) \mid (U,c_U) \in \cat{M}_{c_X}\}
	\end{align*}
	where $\mcM_{c_X}/c_U$ is the slice category of $\cat{M}_{c_X}$ over $(U,c_U)$. Then $\site{M}_{c_X}:= (\cat{M}_{c_X},\cov(\cat{M}_{c_X}))$ is a site on $\cat{M}_{c_X}$.
\end{theorem}

\begin{proof}
	First, we have that any inclusion $\iota^V_U:(V,c_V)\hookrightarrow (U,c_U)$ which is a homeomorphism $(V,c_V) \xrightarrow{\cong} (U,c_U)$ is itself an
	interior cover of $(U,c_U)$, so Definition \ref{def:Site}(\ref{item:Site 1}) is satisfied.
	
	To see Definition \ref{def:Site}(\ref{item:Site 2}), first note that, by \autoref{lem:Pullbacks}, $(V \times_U U_{\alpha},c_{V\times_U U_\alpha}) =
	(V  \cap U_{\alpha},c_{V\cap U_{\alpha}})$. The conclusion now follows from
	\autoref{prop:Conditions on covers}(\ref{item:Conditions on covers 1}).
	
	Finally, Definition \ref{def:Site}(\ref{item:Site 3}) follows from \autoref{prop:Conditions on
		covers}(\ref{item:Conditions on covers 2}).
\end{proof}

\subsection{The Grothendieck Topology on $\cat{M}_{c_X}$}

In this section, we recall how a site generates a Grothendieck topology and identify the Grothendieck topology generated by the site $\site{\cat{M}}_{c_X}$. We begin by reviewing the definition of sieves and Grothendieck topologies on a category, after which we recall how a site determines a Grothendieck topology.

Let $\mathcal{C}$ be a category, and, as before, for any $U \in \tmmathbf{\tmop{Ob}}
(\mathcal{C})$, we let $\mathcal{C}/U$ denote the slice category of 
$\cat{C}$ over $U$.

\begin{definition}
  Let $U \in \Ob (\mathcal{C})$. A {\emph{sieve}} $S$ over $U$ is a subset of
  $\Ob (\cat{C}/U)$ such that, if $V \to U \in S$, then the composition $W \to V
  \to U \in S$ for any $W \to V \in \tmop{Hom}_{\mcC} (W, V)$.
\end{definition}

\begin{definition}[\cite{stacks-project}, \href{https://stacks.math.columbia.edu/tag/00Z4}{00Z4}]
  A {\emph{Grothendieck topology}} $J = \{ \mcS (U)\}_{U \in \Ob (\mcC)}$ on a
  category $\mcC$ is a collection of sieves $\mcS (U)$ for each object $U \in
  \Ob (\mcC)$, such that
  \begin{enumerate}
    \item The maximal sieve on $U$ is in $\mcS (U)$, i.e. $\Ob (\mcC_U) \in \mcS
    (U)$
    
    \item Let $V \to U \in \Hom_\mcC(V,U)$. If $S \in \mcS (U)$, then $S
    \times_U V \in \mcS (V)$, where we define
    \[ S \times_U V \assign \{W \to V \mid \text{the composition } W \to V \to
       U \in \mcS (U)\} \]
    \item Let $S$ and $S'$ be sieves over $U$. Assume that $S' \in \mcS (U)$
    and that $S \times_U V \in \mcS (U)$ for any $(V \to U) \in S'$. Then $S \in
    \mcS (U)$.
  \end{enumerate}
  A sieve $S$ over $U$ is called a {\emph{covering sieve of $J$}} if $S \in \mcS
  (U)$.
\end{definition}

Given a site $\mcC$, we construct the Grothendieck topology generated by the site in the following way. We start with the following definition.

\begin{definition}
	Given $\{f_i:U_i\to U\}_{i\in I}$ a family of elements of $\mathcal{C}_U$ (morphisms of $\mcC$ with target $U$). We define the \emph{sieve $S$ on $U$ generated by the morphisms $f_i$} to be the 
	collection of morphisms $g:V\to U$ which factor through one of the $f_i$, i.e. such that there exists a morphism $g':V \to U_i$ with $g = f_i \circ g': V \xrightarrow{g'} U_i \xrightarrow{f_i} U$.
\end{definition}

The Grothendieck topology generated by a site is given by the following definition.

\begin{definition}[\cite{stacks-project}, \href{https://stacks.math.columbia.edu/tag/00Z5}{Lemma 00ZC},
\href{https://stacks.math.columbia.edu/tag/00ZD}{Definition 00ZD}]
	Let $\site{C}$ be a site with coverings $\cov(\cat{C})$. The topology associated to $\site{C}$ is the topology $J$ described by the following: For every object $U$ of $\cat{C}$, we let $J(U)$ be the set of all sieves $S$ such that there exists a covering $\{f_i:U_i \to U\}_{i \in I} \in \site{C}(U)$ for which the sieve $S'$ generated by the $f_i$ is contained in $S$.
\end{definition}

In other words, every sieve $S \in J(U)$ in the topology $J$ associated to a site $\site{C}$ contains a sieve $S'$ generated by a covering $\{f_i: U_i \to U\}_{i\in I} \in \site{C}(U)$ in the $\site{C} = (\cat{C},\cov(\cat{C})$.

\begin{definition}\label{def:Topology generated by site} Given a closure space $(X,c_X)$, we denote by $J_X$ the topology generated by the site $\site{M}_{c_X}$.
\end{definition}

\subsection{Sheaves on the Site $\site{\cat{M}}_{c_X}$}

In this section, we recall the definition of a sheaf on a site, state an equivalent formulation for sheaves on the site $\site{\cat{M}}_{c_X}$ in terms of sections of interior covers of a \v Cech closure space $(X,c_X)$, and briefly indicate how the functors $f_*, f^*, \otimes,$ and $\hom$ are defined for sheaves on $\cat{M}_{c_X}$. 

We begin with the definition of a sheaf on a site.

\begin{definition}[\cite{stacks-project}, \href{https://stacks.math.columbia.edu/tag/00VM}{00VM}]
	\label{def:Sheaves}
	
	Let $\site{C} = (\cat{C},\cov(\cat{C}))$ be a site. A \emph{presheaf on $\site{C}$ with values in $\mathcal{A}$} is a functor $F:\cat{C}^{\text{op}} \to \cat{A}$. A presheaf $F$ on $\site{C}$ is a \emph{sheaf} iff for any cover $\{U_i \to U\}_{i \in I} \in \cov(\cat{C})$, the diagram
	\[
	\begin{tikzcd}[column sep=large] F(U)\arrow[r,"(\phi_i)_{i\in I}"] &  \prod F(U_i) \arrow[r,"(\phi_{ji})_{i,j\in I}",shift left] \arrow[shift right]{r}[swap]{(\phi_{ij})_{i,j\in I}}& \prod F(U_i \times_U U_j)\end{tikzcd}
	\]   
	is an equalizer, where the $\phi_i:F(U)\to F(U_i)$ and $\phi_{ij}:F(U_i) \to F(U_i \times_U U_j)$ are the respective restriction maps.
\end{definition}

The following proposition gives an equivalent formulation of this definition as a generalization of the more classical definition of sheaves on topological spaces.

\begin{proposition}	
	\label{prop:Closure sheaf definition}
	Let $(X,c_X)$ be a closure space. A presheaf $F:\cat{M}_{c_X}^{\text{op}} \to \mathcal{C}$ on the site $(\cat{M}_{c_X},J)$ is a sheaf iff, for every interior cover $\mathcal{U} = \{ U_i \to X \mid i \in I\}$ of $(X,c_X)$, and every collection of sections $\{s_{U_i} \in F(U_i) \mid i \in I\}$ such that the restriction maps $\rho^{U_i}_{U_{ij}}, \s\rho^{U_j}_{U_{ij}}$ agree on every $U_{ij} = U_i \cap U_j$, then there is a unique global section $s_X \in F(X)$ such that $\rho^X_{U_i}(s_X) = s_{U_i}$ for every $U_i \in \mcU$.
\end{proposition}

\begin{proof} Immediate from Definition \ref{def:Sheaves} and the construction of the site $\site{\cat{M}}_{c_X} = (\cat{M}_{c_X},\cov(\cat{M}_{c_X})$ \autoref{thm:Grothendieck site full}
\end{proof}

\begin{remark} Given any \v Cech closure space $(X,c_X)$, the Grothendieck topology $J_X$ generated by the site constructed in \autoref{thm:Grothendieck site full} provides a canonical Grothendieck topology on the category $\cat{M}_{c_X}$. Sheaves may now be constructed via sheafification of any presheaf  on $\cat{M}_{c_X}$. We refer the reader to \cite{MacLane_Moerdijk_1994}, Sections III.4-5, or \cite{Kashiwara_Schapira_2006}, Sections 17.3-4 for a full description of the sheafification construction.
\end{remark}

\subsubsection{Operations on sheaves and presheaves}
We now briefly discuss four of the six Grothendieck functors for sheaves on closure spaces, namely $f^*,f_*,\otimes,$ and $\sheafhom$.

Let $(X,c_X)$ and $(Y,c_Y)$ be closure spaces. Given a continuous function $f:(X,c_X) \to (Y,c_Y)$, $f$ induces a morphism of the sites $\tilde{f}:\site{M}_{c_Y} \to \site{M}_{c_X}$ by $\tilde{f}(V) = f^{-1}(V)$ for any $V \in Y$. (We leave it to the reader to check that this is a morphism of sites, although it will be enough for the moment that $\tilde{f}$ merely be a functor.) We now define the functor $\tilde{f}^p:\catname{PSh}(\cat{M}_{c_X}) \to \catname{PSh}(\cat{M}_{c_Y})$ between categories of abelian presheaves by $\tilde{f}^p F(V) = F(\tilde{f}(V)) = F(f^{-1}(V))$ on objects, and a morphism of presheaves $v:F \to F'$ in $\catname{PSh}(\cat{M}_{c_X})$, induces a morphism $\tilde{f}^p v:\tilde{f}^p F \to \tilde{f}^p F'$ in $\catname{PSh}(\cat{M}_{c_Y})$ by $(\tilde{f}^p v)_{V} \tilde{f}^pF(V) = \tilde{f}^p (v_{\tilde{f}(V)}(F(\tilde{f}(V))) = v_{f^{-1}(V)}(F(\tilde{f}^{-1}(V)))$. It is a classical fact that $\tilde{f}^p$ has a left adjoint, which we call $\tilde{f}_p$ (see, for instance, \cite{Tamme_1994}, Theorem 2.3.1 for details).

We use the functors $\tilde{f}_p \dashv \tilde{f}^p$ to define the adjoint functors $\tilde{f}_s:\catname{Sh}(\site{\cat{M}}(c_Y))) \to \catname{Sh}(\cat{M}(c_X))$ 
and $\tilde{f}^s:\catname{Sh}(\cat{M}(c_X))) \to \catname{Sh}(\cat{M}(c_Y))$ in the following way. For a category $\cat{C}$, let $i_{\cat{C}}:\catname{Sh}(\cat{C}) 
\hookrightarrow \catname{PSh}(\cat{C})$ denote the inclusion of the category $\catname{Sh}(\cat{C})$ of sheaves on $\cat{C}$ into the category 
$\catname{PSh}(\site{\cat{C}})$ of presheaves on $\cat{C}$, and we let $\sheafify_\cat{C}:\catname{PSh}(\cat{C}) \to \catname{Sh}(\cat{C})$ be the 
sheafification functor, which (by \cite{Tamme_1994}, Theorem 3.1.1) is the left adjoint of $i_\cat{C}$. We now define
\begin{align*}
	\tilde{f}_s & \coloneqq \sheafify_{\cat{M}_{c_X}} \circ \tilde{f}_p \circ i_{\cat{M}_{c_Y}}\\
	\tilde{f}^s & \coloneqq \sheafify_{\cat{M}_{c_Y}} \circ \tilde{f}^p \circ i_{\cat{M}_{c_X}}
\end{align*} 
Finally, as noted in \cite{Tamme_1994}, Example 3.6.4, when $(X,c_X), (Y,c_Y)$ are topological spaces, $f_* = \tilde{f}^s$ and $f^* = \tilde{f}_s$, where $f_*$ and $f^*$ are defined as in \cite{Godement_1958}, Chapter II, 1.12-1.13. We now take the equalities $f_* = \tilde{f}^s$ and $f^* = \tilde{f}_s$ as the definitions of $f^*$ and $f_*$ for sheaves on closure spaces.

For the functors $\otimes$ and $\sheafhom$, the definition is as in the topological case.

\begin{example}
	Let $Q = (V,E,s,t)$ be a quiver, and let $F:\cat{M}^{\text{op}}_{c_G} \to \cat{C}$ be an abelian sheaf on the closure space $(V,c_G)$ induced by the graph $G = (V,E_Q)$, as defined in Example \ref{ex:Closure spaces}\ref{ex:Quiver 1}. Let $(E,c_Q)$ and $(E,c_{Q,-})$ closure spaces on the edges of the quiver as in Examples \ref{ex:Closure spaces}\ref{ex:Quiver 1} and \ref{ex:Closure spaces}\ref{ex:Quiver 2}. Then $s^*F,t^*F$, and $s^*F \otimes t^* F$ are sheaves on $(E,c_Q)$. Similarly, let $G:\cat{M}_{c_Q}^{op} \to \cat{C}$ and $G_-:\cat{M}_{c_{Q,0}}^{op} \to \cat{C}$ be sheaves on $(E,c_Q)$ and $(E,c_{Q,-})$, respectively. Then $s_*G, t_*G, s_* G_-, t_* G_-, s_*G \otimes t_* G$, and $s_* G_- \otimes s_* G_-$ are sheaves on $(V,c_Q)$.
\end{example}

\section{Sheaf Cohomology on $\site{\cat{M}}_{c_X}$}
\label{sec:Cohomology}
In this section, we recall the definition of sheaf cohomology, and we show that, for a topological space $(X,\tau_X)$, there is a natural morphism of sites between the site of open sets $\site{\mcO}_{\tau_X}$ and the site $\site{\mcM}(c_{\tau_X})$ of the topological space seen as a closure space, and, furthermore, we show that this morphism induces an isomorphism in sheaf cohomology. Finally, we recall several results which follow from analyzing the spectral sequence from \v Cech to sheaf cohomology, and we use these to produce examples of non-topological \v Cech closure spaces with non-trivial sheaf cohomology in dimensions one and two.

For the remainder of the article, we assume that all sheaves take values in an abelian category $\cat{A}$ with enough injectives.

\subsection{Sheaf Cohomology}
\label{subsec:Sheaf Cohomology} The cohomology of a closure space $(X, c_X)$ with coefficients in a sheaf
$\s{A}$ is now defined as follows.

\begin{definition}
	Let $\catname{Sh} (c_X,\mcA)$ denote the category of sheaves on $(X,
	c_X)$ with values in an abelian category $\mathcal{A}$ with enough injectives, and let $\s{A} : \mcM (X,
	c)^{\text{op}} \to \mathcal{A}$ be a sheaf on $(X, c_X)$. Let $\Gamma_X : \catname{Sh} (X, c_X) \to \mcA$ denote the global section functor on $(X,c_X)$. We
	define the $q$-th cohomology group of $(X, c_X)$ with coefficients in $\s{A}$
	to be the functor
	$H^q  (X ; \s{A}) \assign R^q \Gamma_X (\s{A}),$
	where $R^q\Gamma_X$ is the $q$-th right derived functor of the global section functor $\Gamma_X$.
\end{definition}

\begin{remark} As in the topological case, it follows from general results of homological
algebra \cites{Godement_1958,Tamme_1994} that the cohomology groups $H^{\ast}  (X ; \s{A})$ may be computed
from $\Gamma (\s{L}^{\ast})$, where $\s{L}^{\ast}$ is an acyclic resolution of
$\s{A}$. 
\end{remark}

We now proceed to show that the sheaf cohomology on a topological space $(X,c_\tau)$ is essentially independent of whether it is constructed using sheaves on the site $\site{\cat{O}}_\tau$ of open covers or on the site $\site{\cat{M}}_{c_X}$ of interior covers.  We first recall the following theorem from \cite{Tamme_1994}.

\begin{theorem}[\cite{Tamme_1994}, Corollary 3.9.3] \label{thm:Comparison} Let $i:\site{C}'\to \site{C}$ be a morphism of sites such that
	\begin{enumerate}
		\item The functor $i$ is fully faithful, and
		\item For any $U' \in \cat{C}'$ and each covering $\{f_j:V_j \to i(U')\}_{j \in J} \in \cov(\cat{C})$, there exists a covering $\{g_{j'}:U'_{j'} \to U'\}_{j'\in J'} \in \cov(\cat{C}')$ such that $\{i(U'_{j'}) \to i(U')\}_{j'\in J'} \in \cov(\cat{C})$ refines $\{f_j:V_j \to i(U')\}_{j \in J}$.
	\end{enumerate}
	Suppose that $U' \in \site{C}'$. For all abelian sheaves $\s{A} $ on $\site{C}$ and $\s{A}'$ on $\site{C}'$, we have
	\begin{align*}
		H^*_{\site{C}'}(U',i^s\s{A}) & \cong H^*_{\site{C}}(i(U'),\s{A}) \text{ and} \\
		H^*_{\site{C}'}(U',\s{A}') & \cong H^*_{\site{C}}(i(U'),i_s \s{A}'). 
	\end{align*}
	
\end{theorem}

Our theorem now follows easily.

\begin{theorem}
	Let $(X,c_\tau)$ be a topological closure space. Denote by $H^q_\mcO(X;F)$ the sheaf cohomology of the topological space $(X,\tau)$ on the site of open sets $\mcO$ with open covers $K_\mcO$. Denote by $f:\mcO(X,\tau) \hookrightarrow \mcM(X,c_\tau)$ the inclusion of sites. For any sheaf $\s{G}$ on the site $(\cat{M},\cov(\cat{M}_{c_X}))$, $H^q_\mcO(X;f^* \s{G}) \cong H^q(X;\s{G})$ and for any sheaf $\s{F}$ on the site $(\mcO,K_\mcO)$, $H^q_\mcO(X,\s{F}) \cong H^q(X,f_* \s{F})$.
\end{theorem}

\begin{proof}
	 Every open cover is an interior cover, and every interior cover on a topological space is refined by the open cover formed by the interiors of the sets in the interior cover, so the hypotheses of \autoref{thm:Comparison} are satisfied. The result now follows from \autoref{thm:Comparison}.
\end{proof}

\subsection{{\v C}ech Cohomology}

We now recall the definition of {\v C}ech cohomology with coefficients in an abelian
presheaf, given a site.
We then use this to provide a class of examples of non-topological
closure spaces induced by graphs for which the sheaf cohomology of the constant sheaf $\underline{\Z}$ is
non-trivial in dimension two, which, in particular, demonstrates that the sheaf cohomology of a graph
viewed as a closure space is different to its sheaf cohomology viewd as a topological space.

\begin{definition}
	Let $(X, c_X)$ be a {\v C}ech closure space, let $F:\cat{M}_{c_X}^{\text{op}} \to \cat{C}$ be an abelian presheaf on $(X, c_X)$, and suppose that
	$\mathcal{U}$ an interior cover of $(X, c_X)$. We define
	\[ H^0  (\mathcal{U}, F) \assign \ker \left( \prod_{U_{\alpha} \in
		\mathcal{U}} F (U_{\alpha}) \rightrightarrows \prod_{U_{\beta},
		U_{\gamma} \in \mathcal{U}} F (U_{\beta} \times_X U_{\gamma}) \right) \]
	For each $q > 0$, we define the $q$-th {\v C}ech cohomology group of the
	cover $\mathcal{U}$ of $(X, c_X)$ with coefficients in $F$ by
	\[ H^q  (\mathcal{U}; F) \assign R^q H^0  (\mathcal{U}, F), \]
	the $q$-th right derived functor of $H^0  (\mathcal{U}, \cdot)$ applied to
	$F$.
\end{definition}

We now recall that the cohomology groups $H^q(\cat{U};F)$ of an interior cover
$\mathcal{U}$ with coefficients in the presheaf $F$ may be identified with the cohomology of the following complex.

\begin{definition}
	Let $(X, c_X)$ be a {\v C}ech closure space, $\mathcal{U}= \{U_{\alpha}
	\}_{\alpha \in A}$ be an interior cover on $(X, c_X)$, and suppose that $F$ is
	an abelian presheaf on $(X, c_X)$. Let
	\[ U_{\alpha_0, \ldots, \alpha_q} \assign U_{\alpha_0} \cap \cdots \cap
	U_{\alpha_q} . \]
	For each integer $q \geq 0$, we define
	\[ C^q (\mathcal{U}, F) \assign \prod_{(\alpha_0, \ldots, \alpha_q) \in A^{q
			+ 1}} F (U_{\alpha_0 \cdots \alpha_q}) . \]
	We additionally define the codifferential $d^q : C^q (\mathcal{U}, F) \to
	C^{q + 1} (\mathcal{U}, F)$ by
	\[ (d^q s)_{i_0, \ldots, i_{q + 1}} = \sum_{k = 0}^{q + 1} (- 1)^k F
	(U_{i_0, \ldots, \hat{i}_k, \ldots, i_{q + 1}}) . \]
\end{definition}

Since $d^2 = 0$, $C^{\ast} (\mc{U}, F)$ is a cochain complex. Its homology is
given by the following theorem.

\begin{theorem}
	[{\cite{Tamme_1994}}, Theorem 2.2.3] For every abelian presheaf $F$ on a {\v
		C}ech closure space $(X, c_X)$, and for every interior cover $\mathcal{U}$ of
	$(X, c_X)$, the group $H^q  (\mathcal{U}, F)$ is canonically isomorphic to the
	$q$-th cohomology group of the complex $C^{\ast} (\mathcal{U}, F)$.
\end{theorem}

The interior covers of $(X, c_X)$ form a directed set, where we write $\cal{U} <
\cal{V}$ iff $\cal{V}, \cal{U}$ are interior covers of $(X, c_X)$ and $\cal{V}$
refines $\cal{U}$. Furthermore, if $\cal{U} < \cal{V}$, there exists a
well-defined homomorphism $H^{\ast}  (\mathcal{U} ; F) \to H^{\ast}  (\mathcal{V} ;F)$, and we may therefore make the following definition.

\begin{definition}
	$\check{H}^q  (X ; F) \assign \underrightarrow{\lim} \,H^q  (\mathcal{U}; F)$,
	where the limit is taken over the the directed set of interior covers $\mathcal{U}$ of $X$.
	$\check{H}^{\ast}  (X ; F)$ is called the {\emph{{\v C}ech cohomology of
			$(X, c_X)$ with coeffecients in the presheaf $F$.}}
\end{definition}

We now recall several results which we will use to compute the sheaf cohomology of our examples below.

\begin{theorem}
	[{\cite{Tamme_1994}}, Corollary 3.4.6]\label{thm:Edge morphisms 1}For all
	abelian sheaves $F$ on a closure space $(X, c_X)$, the homomorphism
	\[ \check{H}^p  (X; F) \to H^p  (X, F) \]
	is a bijection for $p = 0, 1$ and an injection for $p = 2$.
\end{theorem}

\begin{theorem}[\cite{Tamme_1994}, Corollary 3.4.7]\label{thm:Edge morphisms 2}
	Let $\mathcal{U} = \{U_i \to X\}_{i\in I}$ be an interior covering of the closure space $(X,c_X)$, and let $F$ be
	an abelian sheaf such that $H^q(U_{i_0} \times_X \cdots \times_X U_{i_q}; F) = 0$ for all $q > 0$ and
	all $(i_0,\dots,i_q) \in I^{q+1}$. Then
	\begin{equation*}
		H^p(\mathcal{U};F) \cong H^p(X,F)
	\end{equation*}
	for all $p$.
\end{theorem} 

We now use the above to give an example of a non-topological closure space with non-trivial sheaf cohomology in dimensions one and two. We denote by $\underline{\Z}$ the constant sheaf given by the sheafification of the constant presheaf $\Z$. We start with the following definitions.

\begin{definition}
	Let $(X,c_\tau)$ be a topological closure space, i.e. such that with $c_\tau^2 = c_\tau$, and suppose that $\cat{U} = \{U_i \subset X\}_{i \in I}$ is an interior cover of $(X,c_X)$. We say that the cover
	$i(\cat{U}) \coloneqq \{i_\tau(U) \mid U \in \cat{U}\}$ is the \emph{open refinement of $\cat{U}$}. 
\end{definition}

\begin{definition}
	A topological space $X$ is \emph{locally contractible} iff any open subset $U \subset X$ has an open cover $\{U_i\}_{i\in I}$ by open
	subsets $U_i \subset U$ which are contractible in $U$.
\end{definition}

\begin{remark}
	Note that the open refinement $i(\cat{U})$ of an interior cover $\cat{U}$ on a topological space is an open cover.
\end{remark}

\begin{lemma}
	\label{lem:Mapping lemma}
	Let $(X,c_X)$ be a closure space, and let $F:\cat{M}^{\text{op}}_{c_X} \to
	\cat{A}$ be a constant abelian presheaf on $(X,c_X)$. Suppose that
	$(Y,c_\tau)$ is a paracompact topological Hausdorff space.
	Suppose $(X,c_X)$ has a maximal interior cover $\mathcal{U}_X$
	(where interior covers are partially ordered by refinements),
	and suppose that there exists a continuous map $f:(Y,c_\tau)
	\to (X,c_X)$ such that the open refinement of
	$f^{-1}(\mathcal{U}_X)$ is a good cover on
	$(Y,c_Y)$ and that the induced map
	$f^*:C_X^*(\mathcal{U}_X,F) \to C_Y^*(i^{-1}\circ
	f^{-1}(\mathcal{U}_X),i^*f^*F)$ is a quasi-isomorphism.
	Then $\check{H}^*(X,F) \cong \check{H}^*(Y,i^*f^*F) \cong H(Y,i^*f^*F)$.         
\end{lemma}

\begin{proof}
	Since $f^*$ is a quasi-isomorphism by hypothesis, we have 
	$H_X^*(\mathcal{U}_X;F) \cong H_Y^*(i^{-1}\circ
	f^{-1}(\mathcal{U}_X);i^*f^*F)$. However, $\mathcal{U}_X$ is a maximal interior cover on $X$ and $i^*f^*\mathcal{U}_X$ is a good cover by hypothesis. Furthermore, $i^*f^*F$ is constant since $F$ is constant. We therefore have
	\begin{align*}
		\check{H}^*(X;F) &\cong H^*(\mathcal{U}_X;F) \cong  H^*(i^{-1}f^{-1}(\mathcal{U}_X);i^*f^*F) \cong \check{H}^*(Y;i^*f^*F) \\
		& \cong H(Y,i^*f^*F),
	\end{align*}
	where the last isomorphism follows from \autoref{thm:Edge morphisms 2}, in which the fact that $Y$ is a paracompact Hausdorff space and the sheaf $i^*f^*F$ is constant guarantees that the hypotheses of \autoref{thm:Edge morphisms 2} are satisfied.
\end{proof}

\begin{lemma}
	\label{lem:Torus lemma}
	Let $\Z_n \coloneqq \Z/n\Z$, and endow $\Z_n$ with the ``nearest neighbor'' closure structure $c_{\Z_n}(k) = \{k-1,k,k+1\} \mod n$. If $n\geq 6$, then $\check{H}^*(\Z^k_n;\underline{\Z}) \cong H_{sing}^*(\T^k;\Z)$, where $\T^k$ is the $k$-dimensional topological torus and $H_{sing}^*$ is singular cohomology.
\end{lemma}

\begin{proof}	
	Let $f:S^1 \to \Z_n$ be the map $f(x) = i$ for $x \in [i-1/(2n),i+1/(2n))$. Denote by $f^k:\T^k \to \Z^k_n$ the product map $(f,\dots,f)$ from the $k$-dimensional torus to $\Z^k_n$. Let $\mathcal{U}\coloneqq \{\{k-1,k,k+1\} \mod n \}$ denote the maximal interior cover on $\Z_n$ (where the interior covers are ordered by refinements), and define the maximal interior cover on $\Z^k_n$ by $\mathcal{U}^k \coloneqq \{ U_1 \times U_2 \times \cdots \times U_k \mid U_i \in \mathcal{U}\}$. Then $i^{-1}f^{-1}(\mathcal{U}^k)$ is a good cover on $\T^k$, and a finite number of sets $V_1,\dots,V_m \in \mathcal{U}^k$ intersect in $X$ iff the preimages $f^{-1}(V_1) \, \dots,f^{-1}(V_m)$ intersect in $Y$. It follows that the induced map $f^*:C_X^*(\mathcal{U}_X,\underline{\Z}) \to C_\T^*(i^{-1}\circ
	f^{-1}(\mathcal{U}_X),i^*f^*\underline{\Z})$ is an isomorphism and commutes with the codifferential, and therefore induces an isomorphism on homology. By \autoref{lem:Mapping lemma}, $\check{H}^*(\Z^k_n;\underline{\Z}) \cong H^*(\T;\underline{\Z})$. However, since $\T$ is a locally contractible, paracompact, topological Hausdorff space, $H^*(\T;\underline{\Z}) \cong H^*_{sing}(\T;\Z)$, since \v Cech and sheaf cohomology of $\T$ are isomorphic by \cite{Godement_1958}, Theorem 5.10.1, and \v Cech cohomology and singular cohomology of $\T$ are isomorphic by \cite{Spanier_1966}, Corollary 6.9.5 and Corollary 6.8.8. 
	The result now follows.
\end{proof}

Combining these lemmas with \autoref{thm:Edge morphisms 1}, we have

\begin{theorem}
	Let $\Z_n$ and $c_{\Z_n}$ be as in \autoref{lem:Torus lemma}. For any $n\geq 6$, $H^1(\Z^k_n, \underline{\Z}) \cong H_{sing}^1(\mathbb{T}^k,\Z)$ and
	$H_{sing}^2(\T^k,\Z) \hookrightarrow H^2(\Z^k_n,\underline{\Z})$,  where $\mathbb{T}^k$ is the $k$-dimensional topological torus.
\end{theorem}

In particular, since the closure spaces $(\Z^k_n,c_{\Z_n})$ are induced by graphs, this gives a class of examples of undirected graphs which have non-trivial sheaf cohomology in dimension two, demonstrating that the
sheaf cohomology of the closure space induced by a graph may be significantly different from the sheaf cohomology of an undirected graph seen as a topological space.

\section*{Acknowledgements}
We are grateful to Carina Curto, Nikola Milicevic, and Nora Youngs for helpful discussions, and 
we're grateful for the opportunity to present preliminary versions of this work at the Northeastern University and the University of Minnesota Topology Seminars,
the Special Session "Bridging Applied and Quantitative Topology" at the 2024 AMS Joint Mathematics Meeting, as well as at the 2023 ICERM semester program ``Math + Neuroscience: Strengthening the Interplay Between Theory and Mathematics". We are also grateful to Justin Curry for the suggestion to try to examine quivers in the context of the theory presented here.

\printbibliography

\end{document}